\documentclass{amsart}
\usepackage{stmaryrd}
\usepackage{silence}
\usepackage{upgreek}
\usepackage{amssymb}
\usepackage{faktor}
\usepackage[all,cmtip]{xy}
\usepackage{amsmath} 
\usepackage{mathrsfs}
\usepackage{tikz}
\usepackage{tikz-cd}
\usepackage{enumitem}
\usepackage{mathtools}
\usepackage[mathcal]{euscript}
\usepackage{graphicx}
\usepackage{url}
\usepackage{hyperref}
\usepackage[T1]{fontenc}
\usepackage{subfigure}

\usetikzlibrary{shapes.geometric}
\usetikzlibrary{decorations.markings}
\usetikzlibrary{patterns,decorations.pathreplacing}

\hypersetup{colorlinks=true,linkcolor=blue}

\usepackage{ragged2e}

\newcommand{\D}{\mathcal{D}}

\definecolor{coloryellow}{RGB}{240,228,66}
\definecolor{colorskyblue}{RGB}{86,180,233}
\definecolor{colorvermillion}{RGB}{213,94,0}

\newcommand{\cC}{\mathcal{C}}

\DeclareSymbolFont{sfletters}{OT1}{cmss}{m}{n}
\DeclareMathSymbol{\sTheta}{\mathord}{sfletters}{"02}

\theoremstyle{definition}
\newtheorem{definition}{Definition}[section]

\newtheorem{example}[definition]{Example}

\newtheorem{construction}[definition]{Construction}

\theoremstyle{plain}
\newtheorem{proposition}[definition]{Proposition}
\newtheorem{lemma}[definition]{Lemma}
\newtheorem{corollary}[definition]{Corollary}
\newtheorem{theorem}[definition]{Theorem}
\newtheorem{conjecture}[definition]{Conjecture}

\theoremstyle{remark}
\newtheorem{remark}[definition]{Remark}

\makeatletter
\@addtoreset{definition}{section}
\makeatother

\usepackage{marginnote}
    \DeclareFontFamily{U}{wncy}{}
    \DeclareFontShape{U}{wncy}{m}{n}{<->wncyr10}{}
    \DeclareSymbolFont{mcy}{U}{wncy}{m}{n}
    \DeclareMathSymbol{\Sha}{\mathord}{mcy}{"58}

\newsavebox{\foobox}

\newcommand{\TGE}{\mathcal{TGE}}

\newcommand{\GE}{\mathcal{GE}}
\newcommand{\GM}{\mathcal{GM}}
\newcommand{\TGI}{\mathcal{TGI}}
\newcommand{\GTM}{\mathcal{GTM}}
\newcommand{\CFE}{\mathcal{CFE}}

\title{Robertson's conjecture and universal finite generation in the homology of graph braid groups}
\author{Ben Knudsen and Eric Ramos}

\begin{document}

\begin{abstract}
We formulate a categorification of Robertson's conjecture analogous to the categorical graph minor conjecture of Miyata--Proudfoot--Ramos. We show that these conjectures imply the existence of a finite list of atomic graphs generating the homology of configuration spaces of graphs---in fixed degree, with a fixed number of particles, under topological embeddings. We explain how the simplest case of our conjecture follows from work of Barter and Proudfoot--Ramos, implying that the category of cographs is Noetherian, a result of potential independent interest.
\end{abstract}

\maketitle

\section{Introduction}

The homology of configuration spaces of graphs has been a subject of significant interest in recent years. One important guiding principle in this study has been the concept of \emph{universal generation} (see, e.g., \cite[Rmk. 3.14]{AnDrummondColeKnudsen:ESHGBG}). Write $F_n(G)$ for the configuration space of $n$ ordered points in the graph $G$.

\begin{definition}\label{def:ufg}
Fix a class of graphs $\mathcal{G}$ and natural numbers $i$ and $n$. We say that \emph{universal finite generation} holds for $(\mathcal{G}, i, n)$ if there is a finite subset $\mathcal{G}_0\subseteq \mathcal{G}$ such that, for any $G\in\mathcal{G}$, the group $H_i(F_n(G))$ is generated by classes arising from topological subgraphs homeomorphic to members of $\mathcal{G}_0$. If universal finite generation holds for $(\mathcal{G},i,n)$ for every $i$ and $n$, then we say simply that universal finite generation holds for $\mathcal{G}$.
\end{definition}

\begin{conjecture}\label{conjecture:ufg}
Universal finite generation holds for the class of all graphs.
\end{conjecture}

This conjecture can be viewed as a statement of finite generation over a certain category of graphs (Conjecture \ref{conjecture:main} below), a reformulation locating it within the realm of categorical Noetherianity results. Categorifications of Higman's lemma \cite{sam} and Kruskal's tree theorem \cite{Barter,PR-genus} are known, and the categorical graph minor conjecture of \cite{MPR,MR} categorifies the theorem of Robertson--Seymour \cite{RSXX}. In this vein, we formulate a categorification of Robertson's conjecture \cite{LT} (Conjecture \ref{conjecture:crc} below), which we relate to universal finite generation.

\begin{theorem}\label{thm:main}
The categorical graph minor conjecture and the categorical Robertson conjecture together imply Conjecture \ref{conjecture:ufg}.
\end{theorem}

\begin{remark} For simplicity, we work throughout with ordered configuration spaces, but the obvious unordered analogue of Definition \ref{def:ufg} is of equal interest, the analogue of Conjecture \ref{conjecture:ufg} equally plausible, and the analogue of Theorem \ref{thm:main} true, with the same proof.
\end{remark}

Briefly, paired with the pioneering work of Abrams \cite{A}, the categorical Robertson conjecture guarantees universal finite generation for graphs whose complexity is bounded in a certain sense. The theorem follows after showing that the resulting filtation by this notion of complexity is finite; in fact, it suffices to establish finiteness of the related filtration by first Betti number, which follows from the categorical graph minor conjecture. 

The simplest case of the categorical Robertson conjecture is already implicit in work of Barter and Proudfoot--Ramos (see Theorem \ref{thm:genus one} below), and this case implies universal finite generation for a large class of graphs including all complete and complete bipartite graphs (see Definition \ref{def:cograph} below).

\begin{theorem}\label{thm:cograph ufg}
Universal finite generation holds for the class of cographs.
\end{theorem}

Key to this result is a Noetherianity theorem for the category of cographs (Theorem \ref{thm:cographs}), a result of potential independent interest.

To some readers, Conjecture \ref{conjecture:ufg} may appear quite bold. For these readers, it may be helpful to recall that we are exploring finiteness questions for \emph{$i$ and $n$ fixed}---although related questions where this is not so are also interesting. Heuristically, we are supposing that there is a limit to the complexity of an $i$-dimensional shape described by $n$ particles confined to a graph, which perhaps sounds plausible enough. A better reason to believe the conjecture is that it is supported by calculations; for example, for every $n$, universal finite generation is known for $i=1$ and the class of all graphs \cite{KP}, for $i=2$ and the class of planar graphs \cite{AnKnudsen:OSHPGBG}, and for all $i$ and the classes of trees with loops \cite{CL} and wheel graphs \cite{MaciazekSawicki:NAQSG} (we elide the distinction between ordered and unordered configurations).

To other readers, the question of generation by subgraphs may seem limiting---why not allow minors as well? A first answer is that this weaker generation question is already well studied \cite{MPR}. A more substantive answer is that the minor relation is a \emph{combinatorial} feature of graphs, where the subgraph relation is \emph{topological}. It is natural to consider the finite generation question in topological terms because the configuration spaces themselves are topological invariants; indeed, it is far from obvious that there is any functoriality for minors to ask about \cite{ADK}.

\subsection{Conventions} All rings are commutative with unit. Homology is taken with integer coefficients, although this assumption is not essential. All categories are essentially small. We say that a functor $F$ \emph{reflects} a property of an object or an arrow in the source provided the property holds if it holds after applying $F$.

\subsection{Acknowledgements} The present paper grew out of conversations at the AIM workshop ``Configuration Spaces of Graphs'' and the conference ``Compactifications, Configurations, and Cohomology'' at Northeastern University. The first author owes his interest in the problem of universal finite generation to Byunghee An and Gabriel Drummond-Cole. He was supported by NSF grant DMS-1906174. The second author would like to send his thanks to Chun-Hung Liu for conversations related with this project. He was supported by NSF grant DMS-2137628.

\section{Categorical modules}

In this section, we review the elementary categorical module theory used in the remainder of the paper. Throughout this section, we fix a ring $A$.

\begin{definition}
Let $\cC$ be a category. A $\cC$-\emph{module} (over $A$) is a functor from $\cC$ to the category of $A$-modules. A \emph{map of $\cC$-modules} is a natural transformation of functors. 
\end{definition}

The resulting category of $\cC$-modules is Abelian, with all relevant concepts (injections, surjections, direct sums, etc.) defined objectwise.

\begin{example}\label{example:free modules}
Fixing $c_0$, the assignment sending $c\in \cC$ to the free $A$-module generated by the set $\cC(c_0,c)$ extends in an obvious way to a $\cC$-module. We denote this $\cC$-module abusively by $\cC(c_0,\bullet)$.
\end{example}

We refer to the modules of this example as \emph{free modules}.

\begin{definition}
A $\cC$-module is said to be \emph{finitely generated} if it admits a surjection from a finite direct sum of free modules. We say that $\cC$ is \emph{Noetherian over $A$} if submodules of finitely generated $\cC$-modules are again finitely generated.
\end{definition}

Given a functor $\Phi:\cC\to \mathcal{D}$, we write $\Phi^*$ and $\Phi_!$ for the restriction and left Kan extension along $\Phi$, respectively, regarded as adjoint functors between the categories of $\cC$-modules and $\mathcal{D}$-modules. The left Kan extension exists by (the dual of) \cite[Thm. X.3.1]{MacLane:CWM}, since the category of $\mathcal{D}$-modules is cocomplete. As a left adjoint, $\Phi_!$ preserves colimits.

We record the following easy facts about these functors.

\begin{lemma}\label{lem:full preserve surjection}
Let $\Phi:\mathcal{C}\to \mathcal{D}$ be a functor. The functors $\Phi^*$ and $\Phi_!$ preserve surjections. If $\Phi$ is essentially surjective, then $\Phi^*$ also reflects surjections.
\end{lemma}
\begin{proof}
The claims about $\Phi^*$ are trivial. The claim about $\Phi_!$ follows from the fact that it preserves cokernels (which are colimits), hence surjections. 
\end{proof}

\begin{lemma}\label{lem:counit surjection}
Let $\Phi:\mathcal{C}\to \mathcal{D}$ be an essentially surjective functor. For any $\D$-module $M$, the counit $\Phi_!\Phi^*M\to M$ is a surjection.
\end{lemma}
\begin{proof}
By essential surjectivity and Lemma \ref{lem:full preserve surjection}, it suffices to show that the counit is a surjection after applying $\Phi^*$, but this map admits a section by one of the triangle identities for the $(\Phi_!,\Phi^*)$-adjunction.
\end{proof}

\begin{lemma}\label{lem:kan extension fg}
Let $\Phi:\cC\to\mathcal{D}$ be a functor and $M$ a $\cC$-module. If $M$ is finitely generated, then so is $\Phi_!M$.
\end{lemma}
\begin{proof}
Supposing that $M$ is finitely generated, there are objects $c_i\in \cC$ and a surjection
\[
\bigoplus_{i=1}^n \cC(c_i,\bullet)\to M.
\]
The Yoneda lemma implies the existence of the dashed arrow in the commutative diagram
\[\xymatrix{
\displaystyle\bigoplus_{i=1}^n \Phi_!\cC(c_i,\bullet)\ar[d]\ar[r]& \Phi_!M\\
\displaystyle\bigoplus_{i=1}^n \mathcal{D}(\Phi(c_i),\bullet).\ar@{-->}[ur]&
}\] Since the top arrow is surjective by Lemma \ref{lem:full preserve surjection}, the dashed arrow is so as well, implying the claim.
\end{proof}

\begin{lemma}\label{lem:full and essentially surjective}
Let $\Phi:\mathcal{C}\to \mathcal{D}$ be a functor and $M$ a $\mathcal{D}$-module. If $\Phi$ is essentially surjective, then $M$ is finitely generated if $\Phi^*M$ is so. If $\Phi$ is also full, then the converse holds.
\end{lemma}
\begin{proof}
Supposing that $\Phi^*M$ is finitely generated, Lemma \ref{lem:kan extension fg} implies that $\Phi_!\Phi^*M$ is so as well. Thus, by Lemma \ref{lem:counit surjection}, $M$ receives a surjection from a finitely generated module, implying the claim.


Supposing that $M$ is finitely generated, there are objects $d_i\in \mathcal{D}$ and a surjection
\[
\bigoplus_{i=1}^n \D(d_i,\bullet)\to M.
\] By essential surjectivity, we may choose $c_i\in \mathcal{C}$ with $\Phi(c_i)=d_i$, and fullness implies that the first map in the composition 
\[
\bigoplus_{i=1}^n\cC(c_i,\bullet)\to \bigoplus_{i=1}^n\Phi^*\mathcal{D}(d_i,\bullet)\to \Phi^*M
\] is a surjection, implying the claim.
\end{proof}

It will also be useful to speak of finite generation in terms of submodules.

\begin{definition}
Let $M$ be a $\cC$-module and $S$ a set of objects of $\cC$. The submodule of $M$ \emph{generated by} $S$ is the image of the natural map 
\[
\bigoplus_{c\in S}\cC(c,\bullet)\otimes M(c)\to M(\bullet).
\] We say that $M$ is \emph{generated by $S$} if it is equal to the submodule generated by $S$.
\end{definition}

\begin{lemma}\label{lem:object finite generation}
Let $M$ be a $\cC$-module. If $M$ is finitely generated, then $M$ is generated by a finite set of objects. If $M(c)$ is finitely generated for every $c\in \cC$, then the converse holds.
\end{lemma}
\begin{proof}
Suppose first that $M$ is finitely generated, and choose a surjection as in the top row of the following diagram:
\[
\xymatrix{
\displaystyle\bigoplus_{i=1}^n\cC(c_i,\bullet)\ar[r]\ar@{-->}[d]& M\\
\displaystyle\bigoplus_{i=1}^n\cC(c_i,\bullet)\otimes M(c_i).\ar[ur]
}\] By the Yoneda lemma, the data of the map $\cC(c_i,\bullet)\to M$ is equivalent to that of an element of $M(c_i)$, and these elements determine the dashed filler. Since the top map is surjective, so is the diagonal map, implying the claim.

For the converse, choose a generating set of objects $S$ and, for each $c\in S$, a surjection $A\langle X_c\rangle\to M(c)$ from a free $A$-module with each $X_c$ finite. The composite surjection
\[\bigoplus_{c\in S}\bigoplus_{x\in X_c}\cC(c,\bullet)\cong\bigoplus_{c\in S}\cC(c,\bullet)\otimes A\langle X_c\rangle\to \bigoplus_{c\in S}\cC(c,\bullet)\otimes M(c)\to M\] witnesses $M$ as finitely generated.
\end{proof}

We close this section by recording the small fraction of categorical Gr\"{o}bner theory we need---see \cite{sam} for much more. In the following definition, given an object $c$ of the category $\cC$, we write $\cC_c$ for the set of isomorphism classes of objects in the comma category $(c\downarrow\cC)$. If $\cC$ is directed, in that all endomorphisms are the identity, then the set $\cC_c$ is in fact a partial order by putting $f_1 \leq f_2$ so long as there is a morphism $h$ with $f_2 = h \circ f_1$. This relation can be seen to respect isomorphism classes in the comma category.

\begin{definition}\label{def:grobner}
Let $\Phi:\cC\to\mathcal{D}$ be a functor with $\cC$ directed, and consider the following conditions.
\begin{itemize}
\item[(G1)] The set-valued functor $c\mapsto \cC_c$ admits a lift to the category of well-orders.
\item[(G2)] For every object $c\in \cC$, the poset $\cC_c$ is Noetherian in that it admits no infinite descending chains or infinite anti-chains.
\item[(F)] For every object $d\in \mathcal{D}$, the comma category $(d\downarrow \Phi)$ admits a finite weakly initial subset.
\end{itemize}
We say that $\cC$ is \emph{Gr\"{o}bner} if (G1) and (G2) hold. We say that $\Phi$ is a \emph{Gr\"{o}bner cover} if (F) holds and $\cC$ is Gr\"{o}bner.
\end{definition}

\begin{theorem}[Sam--Snowden]\label{thm:samsnowden}
If $A$ is Noetherian, then any category admitting a Gr\"{o}bner cover is Noetherian over $A$.
\end{theorem}

\section{Categories of graphs}

A graph is a finite CW complex of dimension at most $1$. A simple graph is a finite simplicial complex of dimension at most $1$. A topological graph is a topological space homeomorphic to a graph. A subgraph is a subcomplex of a graph. A topological subgraph is a subspace of a graph that is also a topological graph. A path in the simple graph $G$ is finite list of distinct vertices with each consecutive pair adjacent, modulo reversal of order. We write $V_G$, $E_G,$ and $P_G$ for the respective sets of vertices, edges, and paths of $G$. As every edge determines a canonical path, given by its endpoints, we may regard $E_G$ as a subset of $P_G$.

Simple graphs form a category $\GE$ with morphisms injective simplicial maps, which we refer to as embeddings. By definition, an embedding preserves adjacency of vertices; we say it is full if it also reflects adjacency. 

Topological graphs form a category $\TGE$ in which a morphism is a topological embedding, i.e., a homeomorphism onto a topological subgraph. These two categories are related via geometric realization. As configuration spaces of graphs, our main object of study, are isotopy invariant, it will also be natural to consider the quotient category $\TGI$ of topological graphs and isotopy classes of topological embeddings.

Another notion of morphism between graphs will be relevant to our categorification of Robertson's conjecture.

\begin{definition}
Let $G$ and $G'$ be simple graphs. A \emph{topological minor morphism} from $G$ to $G'$ is a pair $\rho = (\rho_V,\rho_E)$ of functions
\begin{align*}
\rho_V: V_G &\rightarrow V_{G'}\\
\rho_E: E_G &\rightarrow P_{G'}
\end{align*}
satisfying the following conditions:
\begin{enumerate}
\item $\rho_V$ is an injection;
\item if $e \in E_G$, then $\rho_E(e)$ has the same endpoints as $e$;
\item if $e \in E_G$ and $v \in V_G$, then $\rho_V(v)\in \rho_E(e)$ if and only if $e$ is incident on $v$;
\item if $e_1\neq e_2 \in E_G$ have the common endpoint $v$, then $\rho_E(e_1)\cap \rho_E(e_2)=\{\rho_V(v)\}$, and otherwise $\rho_E(e_1)\cap\rho_E(e_2)=\varnothing$.
\end{enumerate}
If such a morphism exists, we say that $G$ is a \emph{topological minor} of $G'$.
\end{definition}

\begin{remark}
    Note that the fourth and final condition in the definition of topological minor morphism is necessary and not implied by the prior three. For instance, let $G$ be the graph comprised of two disconnected edges, and let $G'$ be the graph that looks like the letter $X$ with 5 vertices and 4 edges. One can define a map that satisfies the first three conditions above by sending each line segment to one of the two ``legs" of the $X$. In this case, the image paths intersect at the central vertex, which is not in the image of any vertex. In particular, the third condition above is satisfied whereas the fourth is not.
\end{remark}

Alternatively, $G$ is a topological minor of $G'$ if a subdivision of $G$ is homeomorphic to a subgraph of $G$. The wide subcategory of topological minor morphisms $\rho$ such that $\rho_E$ factors through $E_{G'}$ is isomorphic to $\GE$.

Via concatenation of paths, the function $\rho_E$ extends canonically to a function $\rho_P:P_G\to P_{G'}$. Thus, given topological minor morphisms $\rho:G\to G'$ and $\sigma:G'\to G''$, we may define the composite $\sigma\circ\rho=(\sigma_V\circ\rho_V, \sigma_P\circ\rho_E)$, which is easily checked to be a topological minor morphism. We obtain in this way the category $\GTM$ of simple graphs and topological minor morphisms.

Choosing subdivisions of the edges of $G$, a topological minor morphism $\rho:G\to G'$ gives rise to a topological embedding of $G$ into $G'$, which can be made canonical by subdividing evenly. The failure of this construction to respect composition may be corrected by a straight-line isotopy, and there results a well-defined functor $\GTM\to \TGI$. As a matter of terminology, a topological minor morphism that is sent to an isomorphism in $\TGI$ will be called a \emph{subdivision}.

The topological minor relation is to be contrasted with the classical minor relation, where $G$ is said to be a minor of $G'$ if it may be obtained by deleting and contracting edges of $G'$. This relation also corresponds to the existence of a \emph{minor morphism} in a category $\GM$, whose details may be found in \cite{MPR} (where it is called $\mathrm{G}^{\mathrm{op}}$). This category, too, contains the category $\GE$.

We summarize the relationships among the categories discussed so far in the following commutative diagram of functors.

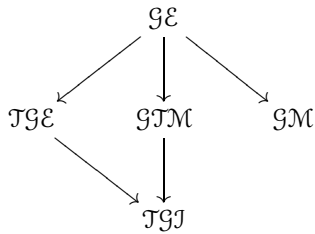
\[
\xymatrix{
&\GE\ar[dl]\ar[d]\ar[dr]\\
\TGE\ar[dr]&\GTM\ar[d]&\GM\\
&\TGI
}\]

We will also have cause to work with graphs with vertices labeled by elements of a well-quasi-order $(Q,\leq_Q)$. In this setting, which we indicate with a subscript $Q$ (e.g., $\GTM_Q$), morphisms are required to be compatible with labels. That is to say, for any vertex $v$, it must be the case that the label of $v$ is $\leq_Q$ than the label of its image.

For future use, we also record a simple, but important, example of a $\GTM$-module.

\begin{example}\label{example:cellular chains}
The assignment $G\mapsto C_*^{CW}(G)$ of a graph to its complex of cellular chains carries a canonical differential graded $\GTM$-module structure. In view of the canonical isomorphism $C_0^{CW}(G)\cong \mathbb{Z}\langle V_G\rangle$, the functoriality in degree $0$ is immediate. In degree $1$, in view of the canonical isomorphism $C_1^{CW}(G)\cong\mathbb{Z}\langle E_G\rangle$, we need only notice the canonical function $P_G\to \mathbb{Z}\langle E_G\rangle$ given by summing over the edges connecting consecutive pairs of vertices in a path. It is an easy exercise to check that the cellular differential is a map of $\GTM$-modules.
\end{example}

\section{Noetherianity conjectures}

According to the graph minor theorem of Robertson--Seymour, the minor relation admits no infinite antichains. The following categorification of this result was conjectured in \cite{MPR}.

\begin{conjecture}[Categorical graph minor conjecture]\label{conjecture:cgmt}
The category $\GM$ is Noetherian over any Noetherian ring.
\end{conjecture}

The topological minor relation, on the other hand, does admit infinite antichains. The example will involve the following important family of graphs.

\begin{definition}
For $k>0$, the $k$th \emph{Robertson chain} is the graph obtained from a path with $k+1$ vertices by doubling each edge and subdividing minimally to achieve simplicity.
\end{definition}

Thus, $R_1$ is a triangle and $R_2$ the union of two triangles at a common vertex.

\begin{example}
Writing $R_k'$ for the graph obtained from $R_k$ by attaching three leaves at each end, the collection $\{R_k'\}_{k\geq1}$ is an antichain for the topological minor relation.
\end{example}

Robertson's conjecture is that this example is the only thing that can go wrong. More precisely, the conjecture states that the topological minor relation admits no infinite antichains after fixing $k>0$ and restricting to graphs not admitting $R_k$ as a topological minor. A proof of this conjecture has been announced by Liu--Thomas \cite{LT}.

Motivated by this conjecture, and writing $\GTM_k\subseteq \GTM$ for the full subcategory spanned by graphs not admitting $R_k$ as a topological minor, we propose the following.

\begin{conjecture}[Categorical Robertson conjecture]\label{conjecture:crc}
Fix $k>0$ and a well-quasi-order $Q$. The category $\GTM_{k,Q}$ is Noetherian over any Noetherian ring.
\end{conjecture}

As we now explain, the simplest case of this conjecture is already implicit in the literature.

\begin{theorem}[Barter, Proudfoot--Ramos]\label{thm:genus one}
Conjecture \ref{conjecture:crc} holds in the case $k=1$.
\end{theorem}
\begin{proof}
Closely related to $\GTM_{1,Q}$ is the category $\mathcal{PGTM}_{1,Q}$ of $Q$-labeled trees equipped with the data of a root vertex and a planar embedding, in which a morphism is a topological minor morphism that preserves both the depth-first ordering induced by the planar structure and the root, as well as the tree partial order induced only by the root, while being compatible with the $Q$-labels as described above. It is not hard to show that the obvious forgetful functor to $\GTM_{1,Q}$ has property (F) of Definition \ref{def:grobner}, so it suffices by Theorem \ref{thm:samsnowden} to verify that $\mathcal{PGTM}_{1,Q}$ is Gr\"{o}bner. In the case where $Q$ is finite and discrete, this fact is \cite[Thm. 3.6]{PR-genus} in somewhat disguised form, and the proof goes through unchanged for general $Q$, since the foundational results of \cite{Draisma} are developed in this generality. For the (contravariant) equivalence between our topological minor morphisms and the contractions of \cite{PR-genus}, see Remark 2.1 of \emph{loc. cit.} and the references therein.
\end{proof}

\section{Configuration spaces of graphs}

We come to our main object of study.

\begin{definition}
Let $X$ be a topological space. For $n\geq0$, the \emph{configuration space} of $n$ (ordered) points in $X$ is the space 
\[
F_n(X)=\{(x_1,\ldots, x_n)\in X^k\mid i\neq j\implies x_i\neq x_j\}.
\]
\end{definition}

An injective map between spaces induces a map between their configuration spaces, and a pointwise injective homotopy induces a homotopy. In particular, the assignment $G\mapsto H_i(F_n(G))$ carries a canonical $\TGI$-module structure and hence, by restriction, canonical $\GTM$- and $\GTM_k$-module structures. Abusively, our notation will fail to distinguish among these structures. 

We now reformulate our main conjecture in these categorical terms.

\begin{conjecture}[Universal finite generation]\label{conjecture:main}
For any natural numbers $i$ and $n$, the $\TGI$-module $H_i(F_n(\bullet))$ is finitely generated.
\end{conjecture}

In this section, we will prove the following related result, deferring the bridge from $\GTM_k$ to $\GTM$ and thence to $\TGI$ to Section \ref{section:filtrations}.

\begin{theorem}\label{thm:truncated fg}
If Conjecture \ref{conjecture:crc} holds, then, for any natural numbers $i$, $k$, and $n$, the $\GTM_k$-module $H_i(F_n(\bullet))$ is finitely generated.
\end{theorem}

Our approach follows a standard two-step strategy. First, realize the module in question as the homology of a differential graded module that is obviously finitely generated. Second, appeal to a categorical Noetherianity statement. Toward the first goal, we recall a cellular model due to Abrams \cite{A}.

\begin{definition}
The $n$th \emph{discretized configuration space} of the graph $G$, denoted $D_n(G)$, is the largest subcomplex of $G^n$ contained in the configuration space $F_n(G)$.
\end{definition}

Explicitly, the open cells of $D_n(G)$ are products of vertices and open edges of $G$ with pairwise disjoint closures.

The utility of this model lies in the following result. As a matter of terminology, we say that a graph satisfying the hypothesis of the theorem is \emph{sufficiently subdivided for $n$}.

\begin{theorem}[Abrams]\label{thm:abrams}
Fix $n > 1$ and a graph $G$. The inclusion $D_n(G)\subseteq F_n(G)$ is a homotopy equivalence provided every path and every cycle in $G$ involves at least $n+1$ edges.
\end{theorem}

At the chain level, Abrams' model fits easily into our categorical framework. Indeed, the assignment $G\mapsto C_*^{CW}(D_n(G))$ canonically extends to a differential graded $\GTM$-module; indeed, one checks esaily that it is a submodule of the module $G\mapsto C_*^{CW}(G)^{\otimes n}\cong C_*^{CW}(G^n)$ (we use the module structure of Example \ref{example:cellular chains}). In particular, we obtain a $\GTM$-action on homology. The following result concerning this module structure is essentially immediate.

\begin{lemma}
For $i,n\geq0$, the canonical map $H_i(D_n(\bullet))\to H_i(F_n(\bullet))$ is a map of $\GTM$-modules.
\end{lemma}

For our purposes, the advantage of the discretized configuration space lies in the following observation.

\begin{lemma}\label{lem:discrete fg}
For any natural numbers $i$, $k$, and $n$, the $\GTM_k$-module $C_i^{CW}(D_n(\bullet))$ is finitely generated.
\end{lemma}
\begin{proof}
Write $G_{i,n-i}$ for the graph with $i$ isolated edges and $n-i$ isolated vertices. By inspection, the action map \[\GTM(G_{i,n-i},G)\otimes C_i^{CW}(D_n(G_{i,n-i}))\to C_i^{CW}(D_n(G))\] is a surjection for every $G\in \GTM$. Since $C_i^{CW}(D_n(G_{i,n-i}))\cong\mathbb{Z}^{i!}$, and since $G_{i,n-i}$ lies in the full subcategory $\GTM_k$, the claim follows.
\end{proof}

\begin{proof}[Proof of Theorem \ref{thm:truncated fg}]
Assuming Conjecture \ref{conjecture:crc}, it follows from Lemma \ref{lem:discrete fg} that the $\GTM_k$-module $H_i(D_n(\bullet))$, as a subquotient of $C_i^{W}(D_n(\bullet))$, is finitely generated. Let $M\subseteq H_i(D_n(\bullet))$ denote the submodule generated by the graphs in $\GTM_k$ that are sufficiently subdivided for $n$. Again invoking Conjecture \ref{conjecture:crc}, we conclude that $M$ is finitely generated, hence generated by the set objects $\{G_j\}_{j=1}^r$ of $\GTM_k$ by Lemma \ref{lem:object finite generation}. 

Choosing subdivisions $\widetilde G_j\to G_j$ with $\widetilde G_j$ a minimal simplicial representative, we claim that the set $\{\widetilde G_j'\}_{j=1}^r$ generates $H_i(F_n(\bullet))$ as a $\GTM_k$-module. To this end, fix an object $G\in \GTM_k$ and a subdivision $G\to G'$ with $G'$ sufficiently subdivided for $n$. We obtain the following commutative diagram:
{\[\xymatrix{
\displaystyle\bigoplus_{j=1}^r\GTM_k(\widetilde G_j,G)\otimes H_i(F_n(\widetilde G_j))\ar[r]\ar[d]&H_i(F_n(G))\ar[d]^-\cong\\
\displaystyle\bigoplus_{j=1}^r\GTM_k(\widetilde G_j,G')\otimes H_i(F_n(\widetilde G_j))\ar[d]_-\cong\ar[r]&H_i(F_n(G'))\\
\displaystyle\bigoplus_{j=1}^r\GTM_k( \widetilde G_j,G')\otimes H_i(F_n(G_j))\ar[r]&H_i(F_n(G')\ar@{=}[d]\ar@{=}[u]\\
\displaystyle\bigoplus_{j=1}^r\GTM_k( G_j,G')\otimes H_i(F_n(G_j))\ar[u]\ar[r]&H_i(F_n(G')\\
\displaystyle\bigoplus_{j=1}^r\GTM_k(G_j,G')\otimes H_i(D_n(G_j))\ar[u]\ar[r]&H_i(D_n(G')).\ar[u]_-\cong
}\]} The lower righthand vertical map is an isomorphism by Theorem \ref{thm:abrams}, since $G'$ is sufficiently subdivided, and the other indicated isomorphisms are induced by subdivisions. Since $G$ is arbitrary, the proof will be complete upon establishing that the top map is a surjection. Since the bottom map is a surjection by our earlier discussion, a diagram chase shows that it suffices to check that the subdivision $G\to G'$ induces a surjection $\GTM_k(\widetilde G_j,G)\to \GTM_k(\widetilde G_j, G')$ for each $j$, which follows from our minimality assumption.
\end{proof}

\section{Geometric filtrations}\label{section:filtrations}

In this section, we close the gap between finite generation over $\TGI$, our desired conclusion, and finite generation over $\GTM_k$ for each $k$, as discussed in the previous section.

\begin{definition}
Fix natural numbers $i$, $n$, and $k$ and $g\geq0$.
\begin{enumerate}
\item The $k$th \emph{Robertson stage} of $H_i(F_n(\bullet))$ is the $\GTM$-submodule $R_kH_i(F_n(\bullet))$ generated by the objects of $\GTM_k$.
\item The $g$th \emph{Betti stage} of $H_i(F_n(\bullet))$ is the $\GTM$-submodule $B_gH_i(F_n(\bullet))$ generated by graphs with first Betti number at most $g$.
\end{enumerate}
\end{definition}

\begin{remark}
There is an obvious analogue of the Betti filtration, in which $H_i(F_n(\bullet))$ is instead considered as a $\TGI$-module. These filtrations coincide.
\end{remark}

\begin{lemma}\label{lem:filtration inclusions}
Fix natural numbers $i$, $n$, and $k$ and $g\geq0$. The following containments hold.
\begin{enumerate}
\item $R_{k}H_i(F_n(\bullet))\subseteq R_{k+1}H_i(F_n(\bullet))$
\item $B_gH_i(F_n(\bullet))\subseteq B_{g+1}H_i(F_n(\bullet))$
\item $B_gH_i(F_n(\bullet))\subseteq R_{g+1}H_i(F_n(\bullet))$
\end{enumerate}
\end{lemma}
\begin{proof}
The first two claims are immediate. For the third, note that the Robertson chain $R_k$ has first Betti number $k+1$, and topological minor morphisms are non-decreasing with respect to this number.
\end{proof}

Thus, we have two ascending, exhaustive filtrations of $H_i(F_n(\bullet))$, one contained in the other. It will be useful to have a second interpretation of one of these. In what follows, we write $\iota_k:\GTM_k\to \GTM$ for the inclusion.

\begin{lemma}\label{lem:robertson image}
The image of the counit $(\iota_k)_!\iota_k^*H_i(F_n(\bullet))\to H_i(F_n(\bullet))$ coincides with $R_kH_i(F_n(\bullet))$.
\end{lemma}
\begin{proof}
The claim is immediate from the standard expression for the left Kan extension as a coequalizer.
\end{proof}

Somewhat surprisingly, as first observed in \cite{ADK}, the homology of configuration spaces of graphs enjoys exceptional functoriality for edge contractions. Concerning this extended functoriality, we have the following result of \cite{MPR}.

\begin{proposition}\label{prop:contractions}
Fix natural numbers $i$ and $n$. The $\GE$-module structure on $H_i(F_n(\bullet))$ extends to a $\GM$-module structure. If Conjecture \ref{conjecture:cgmt} holds, then this $\GM$-module is finitely generated.
\end{proposition}

\begin{corollary}\label{cor:finite filtrations}
If Conjecture \ref{conjecture:cgmt} holds, then the Betti filtration is finite.
\end{corollary}
\begin{proof}
Assuming Conjecture \ref{conjecture:cgmt}, Proposition \ref{prop:contractions} guarantees a $\GM$-module surjection of the form
\[\bigoplus_{j=1}^r\GM(G_j,\bullet)\otimes H_i(F_n(G_j))\to H_i(F_n(\bullet)).\] But every morphism $G_j\to G$ in $\GM$ factors through a morphism $G_j'\to G$ in $\GE$, where $G_j'$ has first Betti number equal to that of $G_j$. Indeed, this follows from the fact that the morphisms of $\GM$ are comprised of edge deletions and contractions, and these operations commute with one another. Thus, the image of the surjection in question lies in the $g$th stage of the Betti filtration, where $g$ is the maximum of the first Betti numbers of the $G_j$. 
\end{proof}

\begin{proof}[Proof of Theorem \ref{thm:main}]
Assuming Conjecture \ref{conjecture:cgmt}, we invoke Corollary \ref{cor:finite filtrations} to conclude the existence of $g_0\geq0$ such that $B_gH_i(F_n(\bullet))=H_i(F_n(\bullet))$ for $g> g_0$. From Lemma \ref{lem:filtration inclusions}, we have the inclusions \[B_{k-1}H_i(F_n(\bullet))\subseteq R_kH_i(F_n(\bullet))\subseteq  H_i(F_n(\bullet)),\] so $R_kH_i(F_n(\bullet))=H_i(F_n(\bullet))$ for $k\geq g_0$. By Lemma \ref{lem:robertson image}, this $\GTM$-module receives a surjection from the module $(\iota_k)_!\iota_k^*H_i(F_n(\bullet))$, which is finitely generated by Lemma \ref{lem:kan extension fg} and Theorem \ref{thm:truncated fg}---here we use the assumed validity of Conjecture \ref{conjecture:crc}. We conclude that $H_i(F_n(\bullet))$ is finitely generated as a $\GTM$-module. The claim now follows from Lemma \ref{lem:full and essentially surjective}, since the functor $\GTM\to \TGI$ is essentially surjective; indeed, every topological graph is homeomorphic to a simple graph.
\end{proof}

\section{Cographs}

As shown above in Theorem \ref{thm:genus one}, the simplest case of the categorical Robertson conjecture holds. In this section, we explore the consequences of this fact for a certain class of graphs.

\begin{definition}\label{def:cograph}
The class of \emph{cographs} is the smallest class of simple graphs satisfying the following three properties:
\begin{enumerate}
\item the isolated vertex is a cograph;
\item a disjoint union of cographs is a cograph; and
\item if $G$ is a cograph, then so is the graph with vertices $V_G$ and edges the non-edges of $G$.
\end{enumerate}
\end{definition}

We write $\CFE\subseteq \GE$ for the subcategory of cographs and full embeddings. Note that this category is neither full nor wide.

\begin{theorem}\label{thm:cographs}
The category $\CFE$ is Noetherian over any Noetherian ring.
\end{theorem}

To prove this result, we employ a standard maneuver relating cographs to another type of graph called a cotree. The reader is warned that a cotree is not a cograph that happens to be a tree. In the following definitions, we write $Q=\{0,1,L\}$, regarded as a discrete quasi-order.

\begin{definition}
A \textbf{cotree} is a $Q$-labeled, rooted tree that is either a singleton labeled by $L$ or else satisfies the following conditions:
\begin{enumerate}
\item the root\footnote{For the purposes of this definition, the root is regarded as an internal vertex and not a leaf.} is labeled by 0 or 1;
\item non-leaf children of vertices labeled by 0 are labeled by 1, and vice versa;
\item leaves are labeled by $L$; and
\item internal vertices have at least two children.
\end{enumerate}
A morphism of cotrees is a topological minor morphism preserving labels (but not necessarily roots). Observe that, because our well-quasi-order is discrete, requiring that our maps preserve the labels is equivalent to those maps being compatible with the labels. In particular, the category of cotrees is a full subcategory of the category of trees with labels in $Q$.
\end{definition}

\begin{proof}[Proof of Theorem \ref{thm:cographs}]
The category of cotrees is Noetherian over any Noetherian ring by Theorem \ref{thm:genus one} and \cite[Prop. 4.4.2]{sam}. Thus, it suffices to show that $\CFE$ is equivalent to the category of cotrees. This fact is essentially standard \cite{D}, but we include an outline of proof at the level of objects for the reader's convenience. 

Given a cotree $T$ with root $v$, we construct a cograph $G_{(T,v)}$ recursively as follows. 
\begin{enumerate}
\item If $(T,v)$ is a singleton, then so is $G_{(T,v)}$.
\item If $v$ is labeled by $0$, then $G_{(T,v)}$ is the disjoint union of the cographs arising from the children of $v$.
\item If $v$ is labeled by $1$, then $G_{(T,v)}$ is the complement of the disjoint union of the complements of the cographs arising from the children of $v$.
\end{enumerate}

Because $G_{(T,v)}$ is formed by a sequence of disjoint unions and complements, it is a cograph. The construction going the other way is more subtle, relying on the fact that every cograph may be constructed uniquely via an alternating sequence of such moves \cite{CLSB}.


\end{proof}

Our next goal is to apply Theorem \ref{thm:cographs} to deduce a universal finite generation result. Unfortunately, because morphisms among cographs are full embeddings, the subdivisions necessary to apply Abrams' theorem as above are missing from the category of cographs. Fortunately, we may avail ourselves of a second cellular model due originally to \'{S}wi\k{a}tkowski.

\begin{construction}\label{construction:cells}
Fix a graph $G$. Denote the set of half-edges of $G$ by $H_G$ and the set of half-edges at the vertex $v$ by $H_G(v)$. Let $A_{i,n}(G)$ denote the set of functions \[\lambda:E_G\sqcup V_G\to \mathbb{Z}_{\geq0}\sqcup V_G\sqcup H_G\sqcup\{\varnothing\}\] satisfying the following conditions: 
\begin{enumerate}
\item $\lambda(E_G)\subseteq \mathbb{Z}_{\geq0}$;
\item $\lambda(v)\in \{\varnothing, v\}\sqcup H_G(v)$ for $v\in V_G$;
\item $\sum_{e\in E_G}\lambda(e)+|\lambda^{-1}(V_G)|+|\lambda^{-1}(H_G)|=n$; and
\item $|\lambda^{-1}(H_G)|=i$.
\end{enumerate}
Extension by the values $0$ and $\varnothing$ endows the assignment $G\mapsto A_{i,n}(G)$ with the structure of a set-valued $\GE$-module.
\end{construction}

The role of the set $A_{i,n}(G)$ is as the set of $i$-cells in a cubical complex with the homotopy type of $F_n(G)/\Sigma_n$.

\begin{theorem}[\cite{Sw-graphs, CL,ADK}]\label{thm:swiatkowski}
Fix $n\geq0$. There is a functor $K_n(\bullet)$ from $\GE$ to the category of equivariant cellular embeddings among cell complexes equipped with free cellular $\Sigma_n$-actions, together with canonical isomorphisms for $i\geq0$ of the following form:
\begin{enumerate}
\item $C_i^{CW}(K_n(\bullet)/\Sigma_n)\cong \mathbb{Z}\langle A_{i,n}(\bullet)\rangle$; and
\item $H_i(K_n(\bullet))\cong H_i(F_n(\bullet))$.
\end{enumerate}
\end{theorem}

We now pair this model with Theorem \ref{thm:cographs} to establish universal finite generation for the class of cographs.

\begin{proof}[Proof of Theorem \ref{thm:cograph ufg}]
We will show that $H_i(F_n(\bullet))$ is finitely generated as a $\CFE$-module, where the module structure is given by restriction along the composite functor \[\CFE\to \GE\to \GTM.\] By Theorems \ref{thm:cographs} and \ref{thm:swiatkowski}, it suffices to establish the claim instead for $C_i^\mathrm{CW}(K_n(\bullet))$. By equivariance, it suffices to establish the claim instead for \[C_i^{CW}(K_n(\bullet))/\Sigma_n\cong C_i^{CW}(K_n(\bullet)/\Sigma_n)\cong \mathbb{Z}\langle A_{i,n}(\bullet)\rangle.\] To this end, fix a cograph $G$. Given a basis element $\lambda\in A_{i,n}(G)$, write $G_\lambda\subseteq G$ for the induced subgraph containing all vertices of the following three types:
\begin{enumerate}
\item members of $\mathrm{im}(\lambda)\cap V_G$;
\item endpoints of edges associated to members of $\mathrm{im}(\lambda)\cap H_G$; and
\item endpoints of members of $\lambda^{-1}(\mathbb{Z}_{>0})$.
\end{enumerate}
Since it is induced, the subgraph $G_\lambda$ is again a cograph, and it is easy to see that the basis element $\lambda$ lies in the image of the homomorphism induced by the inclusion of $G_\lambda$. Moreover, we have that
\begin{align*}
|V_{G_\lambda}|&\leq |\mathrm{im}(\lambda)\cap V_G|+2|\mathrm{im}(\lambda)\cap H_G|+2\sum_{e\in E_G}\lambda(e)\\
&=|\lambda^{-1}(V_G)|+2|\lambda^{-1}(H_G)|+2\sum_{e\in E_G}\lambda(e)\\
&=n+|\lambda^{-1}(H_G)|+\sum_{e\in E_G}\lambda(e)\\
&=n+i+\sum_{e\in E_G}\lambda(e)\\
&\leq 2n,
\end{align*}
where we have used that $\lambda$ is injective on $V_G$ away from the preimage of $\varnothing$, and that $\sum_{e\in E_G}\lambda(e)\leq n-i$, both readily seen from Construction \ref{construction:cells}. We conclude that the $\CFE$-module $\mathbb{Z}\langle A_{i,n}(\bullet)\rangle$ is generated by the collection of cographs with at most $2n$ vertices, which is finite.
\end{proof}

\bibliography{ufg}
\bibliographystyle{amsalpha}

\end{document}